\documentclass[11pt,epsf]{article}

\long\def\comment#1{\ifdim\overfullrule>0pt{\sf[{#1}]}\fi}

\usepackage{latexsym}
\usepackage{amsmath,amsfonts}
\usepackage{bbm}
\usepackage{amsthm,amssymb}
\usepackage[hypertexnames=false]{hyperref}
\usepackage{boxedminipage}
\usepackage{mathtools}
\usepackage{subcaption}
\usepackage{enumitem}
\usepackage{mdframed}
\usepackage{xspace}
\usepackage{mathrsfs}
\usepackage{complexity}
\usepackage{thm-restate}

\usepackage{soul}

\newcommand{\ee}{\mathrm{ee} \xspace}
\newcommand{\nss}{\mathrm{nss} \xspace}
\newcommand{\clust}{\mathrm{clust} \xspace}

\newcommand{\h}{\mathrm{dense} \xspace}
\newcommand{\s}{s \xspace}
\newcommand{\tfcluster}{$(t,\s(t))$-cluster\xspace}

\newcommand{\npm}{N^{\pm}}

\newcommand{\gammaout}{\gamma^+}
\newcommand{\gammain}{\gamma^-}

\newcommand{\klcluster}{$(k,\ell)$-cluster\xspace}

\newtheorem{theorem}{Theorem}

\newtheorem{lemma}[theorem]{Lemma}
\newtheorem{conjecture}[theorem]{Conjecture}
\newtheorem{corollary}[theorem]{Corollary}

\newtheorem{observation}[theorem]{Observation}

\newtheorem{definition}[theorem]{Definition}

\numberwithin{equation}{section}
\numberwithin{theorem}{section}
\numberwithin{lemma}{section}
\numberwithin{conjecture}{section}
\numberwithin{corollary}{section}
\numberwithin{proposition}{section}
\numberwithin{observation}{section}
\numberwithin{definition}{section}
\numberwithin{question}{section}
\numberwithin{fact}{section}

\newcommand{\chiv}{\vec\chi}

\theoremstyle{remark}
\newtheorem{claim}{Claim}

\newenvironment{cproof}
{\begin{proof}
 [Proof.]
 \vspace{-1.5\parsep}
}
{ \end{proof}}

\begin{document} 
\bibliographystyle{alpha}
\def\proofend{\hfill$\Box$\medskip}
\def\Proof{\noindent{\bf Proof:\ \ }}
\def\Sketch{\noindent{\bf Sketch:\ \ }}
\def\eps{\epsilon}

\title{Bounding the chromatic number of dense digraphs by arc neighborhoods}

\author{
Felix Klingelhoefer$^*$ \and Alantha Newman\thanks{Laboratoire
    G-SCOP (CNRS, Univ. Grenoble Alpes), Grenoble, France.  Supported
    in part by ANR project DAGDigDec (ANR-21-CE48-0012).   
 \tt{firstname.lastname@grenoble-inp.fr}}}

\maketitle

\begin{abstract}
The chromatic number of a directed graph is the minimum number of
induced acyclic subdigraphs that cover its vertex set, and
accordingly, the chromatic number of a tournament is the minimum
number of transitive subtournaments that cover its vertex set.  The
neighborhood of an arc $uv$ in a tournament $T$ is the set of vertices
that form a directed triangle with arc $uv$.  We show that if the
neighborhood of every arc in a tournament has bounded chromatic
number, then the whole tournament has bounded chromatic number.  This
holds more generally for oriented graphs with bounded independence
number, and we extend our proof from tournaments to this class of
dense digraphs.  As an application, we prove the equivalence of a
conjecture of El-Zahar and Erd\H{o}s and a recent conjecture of
Nguyen, Scott and Seymour relating the structure of graphs and
tournaments with high chromatic number.
\end{abstract}

\section{Introduction}

The chromatic number of a graph is the minimum integer $k$ required to
partition its vertex set into $k$ independent sets.  The chromatic
number of a tournament (and more generally, a directed graph) is the
minimum integer $k$ required to partition its vertex set into $k$
acyclic sets.  Exploring the similarities and differences between the
two notions is a well-studied
area~\cite{erdos1989ramsey,alon2001ramsey}.

For example, if a graph has a large clique, it must have high
chromatic number.  However, the converse is far from true.  In fact, a
graph can be triangle-free, implying that the neighborhood of each
vertex is an independent set, and yet still have high chromatic
number~\cite{descartes1954k}.  In \cite{berger2013tournaments}, it was
conjectured that this phenomenon does not occur in tournaments.
Specifically, \cite{berger2013tournaments} conjectured that in a
tournament $T$, if each vertex $v \in V(T)$ has an out-neighborhood
$N^+(v)$ that induces a subtournament $T[N^+(v)]$ with bounded
chromatic number, then $T$ itself should have bounded chromatic
number.  This was proved by \cite{harutyunyan2019locToGlobal} with the
following theorem.

\begin{theorem}[\cite{harutyunyan2019locToGlobal}]\label{thm:hltw}
  There is a function $f$ such that if for all $v \in V(T)$,
$\chiv(T[N^+(v)]) \leq t$, then $\chiv(T) \leq f(t)$.
  \end{theorem}

We say a tournament $T$ has vertex set $V(T)$ and arc set $A(T)$.  For
an arc $e=uv \in A(T)$, we define the \emph{neighborhood of arc $e$}
to contain all vertices $w$ in $V(T)$ such that $w$ forms a directed
triangle with $uv$.  Formally, we define $N(e) = N^+(v) \cap N^-(u)$.
A stronger theorem, analogous to Theorem \ref{thm:hltw}, but with
vertex out-neighborhoods replaced by arc neighborhoods, is the
following.

\begin{theorem}\label{thm:ours}
There is a function $f$ such that for any tournament $T$, if for all
$e \in A(T)$, $\chiv(T[N(e)]) \leq t$, then $\chiv(T) \leq f(t)$.
  \end{theorem}

This theorem is a special case of 13.3 in \cite{nguyen2023some}.  We
give a different proof, obtained independently, which we subsequently
extend to prove our main theorem.  Notice that the assumption that
$\chiv(T[N^+(v)]) \leq t$ for every vertex $v \in V(T)$ is stronger
than the assumption that $\chiv(T[N(e)]) \leq t$ for every arc $e \in
A(T)$.  However, our proof of Theorem \ref{thm:ours} uses a theorem
from \cite{harutyunyan2019locToGlobal}, which
they used to prove
Theorem \ref{thm:hltw}.  Thus, we do not give a new proof of Theorem
\ref{thm:hltw}.  We say a tournament $T$ is \emph{$t$-arc-bounded} if
for every arc $e \in A(T)$, $\chiv(T[N(e)])\leq t$.  We can now
restate Theorem \ref{thm:ours} as follows.
\begin{restatable}{theorem}{boundedMain}\label{thm:edge_loc_to_glob}
There is a function $f$ such that for every $t$-arc-bounded
tournament $T$, we have $\chiv(T) \leq f(t)$.
\end{restatable}

We prove Theorem \ref{thm:edge_loc_to_glob} in Section
\ref{sec:locGlobTournaments}, where in addition to the aforementioned
theorem of \cite{harutyunyan2019locToGlobal}, we use ideas from
\cite{klingelhoefer2023coloring}, originally developed to design
efficient algorithms for coloring tournaments.  In Section
\ref{sec:bounded_alpha}, we extend our proof of Theorem
\ref{thm:edge_loc_to_glob} to oriented graphs with bounded
independence number and prove our main theorem.  For the sake of
simplicity, we often refer to oriented graphs as digraphs, but in this
paper, a digraph never contains a directed 2-cycle or ``digon''.
Recall that the {\em independence number} of a digraph is the maximum
size of a vertex set that contains no arcs.  We say a digraph $D$ has
vertex set $V(D)$ and arc set $A(D)$.  As we did for tournaments, for
an arc $e=uv$, we define $N(e) = N^+(v) \cap N^-(u)$.

\begin{restatable}{theorem}{boundedAlphaMain}\label{thm:boundedAlpha}
There is a function $\h$ such that for any digraph $D$ with independence number $\alpha$, if $\chiv(D[N(e)]) \leq t$ for every arc $e \in A(D)$, then $\chiv(D) \leq \h(t,\alpha)$.  
  \end{restatable}

As an application of Theorem \ref{thm:boundedAlpha}, we prove the equivalence of
two conjectures, one on graphs with high chromatic number and one on
tournaments with high chromatic number.  The first one, concerning
graphs, was originally posed by \cite{el1985existence} in the form of
an open problem, which asks if the following conjecture is true.

\begin{conjecture}
[\cite{el1985existence}]\label{conj:ee} For all integers $t,c \geq 1$,
there exists $d \geq 1$, such that if a graph $G$ satisfies $\chi(G)
\geq d$, and has no clique with $t$ vertices (i.e., $\omega(G) < t$),
then there are subsets $A, B \subseteq V(G)$ with $\chi(G[A]),
\chi(G[B]) \geq c$, such that there are no edges between $A$ and $B$.
\end{conjecture}

The second conjecture, concerning tournaments, was recently stated by
\cite{nguyen2023problem}.

\begin{conjecture}[\cite{nguyen2023problem}]\label{conj:big_to_big}
For all $c \geq 0$, there exists $d \geq 0$ such that if $T$ is a
tournament with $\chiv(T) \geq d$, there are two sets $A, B \subseteq
V(T)$ such that $\chiv(T[A]), \chiv(T[B]) \geq c$ and all arcs between
$A$ and $B$ go from vertices of $A$ to vertices of $B$.
\end{conjecture}

\cite{nguyen2023some} show that Conjecture \ref{conj:big_to_big}
implies Conjecture \ref{conj:ee}.  They explore the possibility of the
converse being true, but they do not prove it and write that
Conjecture \ref{conj:big_to_big} seems to be strictly stronger than
Conjecture \ref{conj:ee}.  In Section \ref{sec:equiv}, we prove that
Conjecture \ref{conj:ee} does in fact imply Conjecture
\ref{conj:big_to_big}, showing that the two conjectures are
equivalent.

\setcounter{claim}{0}

\section{Arc local-to-global for tournaments}\label{sec:locGlobTournaments}

In this section, we prove Theorem \ref{thm:edge_loc_to_glob}.  Since
our goal is to color a tournament $T$, we can assume that $T$ is
strongly connected; otherwise $T$ can be partitioned into strongly
connected parts, and each one can be colored separately.  A {\em
  dominating set} (respectively, {\em absorbing set}) in $T$ is a set
of vertices $S \subset V$ such that for every $v \in V\setminus{S}$,
there is a $u \in S$ such that $uv$ (respectively, $vu$) is an arc in
$T$.  By {\em domination number}, we mean the minimum size of a
dominating set.  We will use the following theorem from
\cite{harutyunyan2019locToGlobal}.

\begin{theorem}[\cite{harutyunyan2019locToGlobal}]\label{thm:dom_set}
For every constant $k$, there exist constants $K$ and $\ell$ such that
every tournament $T$ with domination number at least $K$ has a subset
of size $\ell$ that induces a tournament with chromatic number at
least $k$.
\end{theorem}

Following the notation in \cite{aboulker2022heroes}, we define a
\emph{\klcluster} to be a set of vertices $S$ such that $\chiv(T[S])
\geq k$, $|S| \leq \ell$ and $T[S]$ is strongly connected.  This
notion is directly related to the previous theorem, which can be
restated for our purposes as follows.

\begin{corollary}\label{cor:uset}
There exist functions $K$ and $\ell$ such that for every integer $t
\geq 1$, every tournament $T$ contains either i) a dominating set and
an absorbing set, each of size at most $K(t)$, or ii) a
$(t,\ell(t))$-cluster.
\end{corollary}

\begin{proof}
Let $t$ be a constant.  By Theorem \ref{thm:dom_set}, there exist
constants $K(t)$ and $\ell(t)$ such that one can find either a
dominating set of size at most $K(t)$, or a subset of size $\ell(t)$
with chromatic number $t$.  If this subset is not strongly connected,
we can find a strongly connected subset with chromatic number $t$.
Then take the tournament obtained by reversing all the arcs in $T$ and
repeat the previous argument. A dominating set in this tournament is
an absorbing set in $T$, while a subset with high chromatic number
would also have high chromatic number in $T$, as reversing all the
arcs preserves the chromatic number.
\end{proof}

Now let us fix a constant $t$.  Using the function $\ell$ from
Corollary \ref{cor:uset}, we define a \emph{jewel} to be a $(t+1,
\ell(t+1))$-cluster.  We will use the fact that in a $t$-arc-bounded
tournament, for any arc $e$, the vertex set $N(e)$ does not contain a
jewel.\footnote{We will redefine a jewel in Section
\ref{sec:bounded_alpha}, but it will have the same purpose.}  We now
present some useful tools for coloring $t$-arc-bounded tournaments.

\subsection{Jewels and other tools for coloring $t$-arc-bounded tournaments}

We begin with a decomposition lemma for tournaments. 
\begin{lemma}\label{lem:dom-by-path}
Let $T$ be a tournament, and let $P = (v_0, v_1, \ldots, v_k)$ be a
shortest path in $T$ from $v_0$ to $v_k$ with arcs $e_i =
v_{i-1}v_{i}$ for $i:1 \leq i \leq k$.  Then we have the following
properties.
\begin{enumerate}
\item Each vertex in $N^-(v_0) \cap N^+(v_k)$ belongs to $N(e_i)$ for some
  $i: 1 \leq i \leq k$.

  \item If $k \geq 3$, then each vertex in
    $V(P)$ belongs to $N(e_i)$ for some $i:1 \leq i \leq k$.

    \item If $k=2$, then
      $v_0$ belongs to $N(e_2)$ and $v_2$ belongs to $N(e_1)$.
      \end{enumerate}
\end{lemma}

\begin{proof}
First consider a vertex $v$ in $N^-(v_0) \cap N^+(v_k)$ and let $i$ be
the maximum index such that $v \in N^-(v_i)$.  Some $v_i$ must exist,
since $v \in N^-(v_0)$.  Then $v \in N(e_{i+1})$.  Next, consider a
vertex $v \in V(P)$.  Notice that all arcs between vertices $V(P)$
that are not adjacent in $P$ must go backward.  It follows that $v_{i}
\in N(e_{i+2})$ and $v_{i} \in N(e_{i-1})$.  When $k \geq 3$, we can
conclude that each $v_i$ belongs to $N(e_j)$ for some $j$ such that $1
\leq j \leq k$.  When $k = 2$, the same argument applies, except now
$v_1$ belongs neither to $N(e_1)$ nor to $N(e_2)$.
\end{proof}

\begin{lemma}\label{lem:color5}
Let $T$ be a $t$-arc-bounded tournament.  Suppose that $P = (v_0, v_1,
\ldots, v_k)$ is a shortest path from $v_0$ to $v_k$, and let $S =
(N^-(v_0) \cap N^+(v_k)) \cup V(P)$. Then $T[S]$ can be colored with
at most $5t$ colors.
\end{lemma}

\begin{proof}
If $k \geq 3$, then each vertex in $S$ belongs to $N(e_i)$ for some
$i:1 \leq i \leq k$.  For $i < j$, 
we say an arc from a vertex in $N(e_i)$ to a vertex in
$N(e_j)$ is {\em forward} with {\em length} $j-i$.
Observe that there are no arcs from $N(e_i)$ to $N(e_{j})$ for $j \geq i+5$,
since this would give a shorter path from $v_0$ to $v_k$.  Thus, we
can color all the vertices in $S$ using five color palettes of $t$
colors each, using one color palette for each $N(e_i)$ assigned modulo
$5$.
Since all forward arcs have length at most four, each cycle with
vertices belonging to different $N(e_i)$'s has at least two
different colors.  Finally, if $k = 2$, then $T[S]$ can be colored
with $2t+1$ colors, and if $k=1$, then $T[S]$ can be colored with $t +
2$ colors.
\end{proof}

Lemma \ref{lem:color5} can be used to prove the following two lemmas, which we will apply shortly to prove Theorem \ref{thm:edge_loc_to_glob}.

\begin{lemma}\label{lem:KN23}
  Let $T$ be a $t$-arc-bounded tournament, containing 
  two vertices $u$ and
  $v$ such that $\chiv(T[N^+(u)]) \leq g(t)$ and $\chiv(T[N^-(v)])
  \leq g(t)$ for some function $g$.  Then $\chiv(T) \leq 2\cdot g(t) + 5t$.
  \end{lemma}

\begin{proof}
Since $T$ is strongly connected, we can set $v_0 = u$ and $v_k = v$,
find a shortest path from $v_0$ to $v_k$, and apply Lemma
\ref{lem:color5} to color the subtournament $T[S]$.  Any remaining
vertex belongs to at least one of the sets $N^+(v_0)$ and $N^-(v_k)$,
which can each be colored with $g(t)$ colors.
\end{proof}

If $T$ has small dominating and absorbing sets, then the following lemma provides a bound on $\chiv(T)$.

\begin{lemma}\label{lem:dom_bounded}
Let $T$ be a $t$-arc-bounded tournament.  Suppose $T$ has a dominating set
$\gammaout(T)$ and an absorbing set $\gammain(T)$.  Then $\chiv(T)
\leq 5 t \cdot |\gamma^-(T)| \cdot |\gamma^+(T)|$.
\end{lemma}

\begin{proof}
We may assume that $T$ is strongly connected.  Let $q = |\gamma^-(T)|
\cdot |\gamma^+(T)|$.  Let ${\cal P} = \{P_1, P_2, \ldots, P_q\}$ be a
set of $|\gamma^-(T)| \cdot |\gamma^+(T)|$ shortest paths from each $u
\in \gammain(T)$ to each $w \in \gammaout(T)$.  Then for each $v \in
V$, there is some path $P_j \in {\cal P}$ from some $u$ to some $w$
such that $v \in (N^-(u) \cap N^+(w)) \cup V(P_j)$.  So we can apply
Lemma \ref{lem:color5}, which implies the lemma.
\end{proof}

To prove Theorem \ref{thm:edge_loc_to_glob}, we need one more lemma.
First, we give some notation and a definition.  For two disjoint
vertex sets, $X, Y \subset V$, we say $X \Rightarrow Y$ if all arcs
between $X$ and $Y$ go from $X$ to $Y$.  For a set $S \subset V$, we
define the set $\npm(S)$ to be all vertices $v$ in $V\setminus{S}$
such that there exist vertices $u,w \in S$ and arcs $uv$ and $vw$ in
$T$.

\begin{definition}
We define a \emph{jewel-chain} of length $p$ in a tournament to be an
ordered set $X=(X_i)_{1\leq i \leq p}$ such that each $X_i$ induces a
jewel, all $X_i$'s are disjoint, and $X_i \Rightarrow X_{i+1}$ for
all $i$ such that $1 \leq i \leq p-1$.
\end{definition}

For a jewel-chain $X$, we say that an arc $uv$ is {\em forward} if $u
\in X_i$ and $v \in X_j$ and $i < j$.  If $j < i$, then $uv$ is {\em
  backward}.  Jewel-chains in $t$-arc-bounded tournaments are useful because they contain no backward
arcs.

\begin{observation}\label{obs:jewel}
  Let $T$ be a $t$-arc-bounded tournament and let $X$ be a jewel-chain in $T$.  Then $X$ contains no backward arcs.
\end{observation}

\begin{proof}
Consider a backward arc $e = uv$, with $u \in X_j$ and $v \in X_i$ for
$j > i$ such that $j-i$ is minimized.  It must be the case that $j >
i+1$, since all arcs between $X_i$ and $X_{i+1}$ are forward by
definition.  Then $X_{i+1} \subseteq N(e)$, and since $X_{i+1}$ has
chromatic number at least $t+1$, we have $\chiv(T[N(e)]) \geq t+1$,
which contradicts $T$ being $t$-arc-bounded.  Thus, all arcs with
endpoints in distinct $X_i$'s must be forward.\end{proof}

\begin{lemma}\label{lem:main}
Let $T$ be a $t$-arc-bounded tournament that contains a jewel.  Then
there exists a function $g$ such that there are two vertices $u,v$
such that $\chiv(T[N^+(u)]) \leq g(t)$, and $\chiv(T[N^-(v)]) \leq
g(t)$.
\end{lemma}

\begin{proof}
Let $X$ be a jewel-chain in $T$ of maximum length, say $p$.  Consider
$X_1$.  Let $Y$ be the set of vertices such that $Y \Rightarrow X_1$.
Then $Y$ does not contain a jewel (otherwise $X$ would not have
maximum length).  By Corollary \ref{cor:uset}, $Y$ must have a small
dominating set and a small absorbing set, each of size at most
$K(t+1)$.  So we can apply Lemma \ref{lem:dom_bounded} to bound the
chromatic number of $Y$ by $5t \cdot (K(t+1))^2$.  Moreover, the set
$\npm(X_1)$ has chromatic number at most $\ell(t+1) \cdot t$, since
$X_1$ contains a Hamilton cycle with at most $\ell(t+1)$ arcs and each
vertex in $\npm(X_1)$ belongs to $N(e)$ for some $e$ in the Hamilton
cycle.  Finally, a vertex $v$ in $X_1$ can have in-neighbors in $X_1$
itself, but this set has chromatic number at most $|X_1| \leq
\ell(t+1)$.

Set $g(t) = 2\ell(t+1) \cdot t + 5t \cdot (K(t+1))^2$.  Then each
vertex $u \in X_1$ has $\chiv(T[N^-(u)]) \leq g(t)$.  By the same
argument, each vertex $v \in X_p$ has $\chiv(T[N^+(v)]) \leq g(t)$.
This proves Lemma \ref{lem:main}.
\end{proof}

\subsection{Proof of Theorem \ref{thm:edge_loc_to_glob}}

We are now ready to prove Theorem \ref{thm:edge_loc_to_glob}.
\boundedMain*

\begin{proof}
Let $f(t) = 2\cdot g(t) + 5t$, where $g$ is defined as in the end of
the Proof of Lemma \ref{lem:main}.  If $T$ does not contain a jewel,
then by Corollary \ref{cor:uset}, it contains a dominating and an
absorbing set each of size at most $K(t+1)$.  In this case, we can
apply Lemma \ref{lem:dom_bounded} to show that $T$ can be colored with
at most $5t \cdot (K(t+1))^2$ colors.  If $T$ contains a jewel, then
we can apply Lemma \ref{lem:main} and Lemma \ref{lem:KN23} to prove
the theorem.
\end{proof}

\setcounter{claim}{0}

\section{Arc local-to-global for dense digraphs}\label{sec:bounded_alpha}

In this section, we extend Theorem \ref{thm:edge_loc_to_glob} from
tournaments to oriented graphs with bounded independence number.
Since our goal is to color a digraph $D$, we can assume that $D$ is
strongly connected; otherwise $D$ can be partitioned into strongly
connected parts, and each one can be colored separately.  An important
tool for the proof is Theorem \ref{thm:dom_alpha}, which extends
Theorem \ref{thm:dom_set} from tournaments to digraphs with bounded
independence number.\footnote{Theorem \ref{thm:hltw} was extended to
digraphs with bounded independence number by
\cite{harutyunyan2019coloring}, but they did not provide an extension
of Theorem \ref{thm:dom_set}.}  As the proof of Theorem
\ref{thm:dom_alpha} is very similar to the proof of Theorem
\ref{thm:dom_set} from \cite{harutyunyan2019locToGlobal}, it is
deferred to Appendix \ref{app:dom_lemma}.  A $(k,\ell)$-cluster in a
digraph $D$ is a set of vertices $S$ in $V(D)$ such that $\chiv(D[S])
\geq k$, $|S| \leq \ell$ and $D[S]$ is strongly connected.

\begin{restatable}{theorem}{domLemma}\label{thm:dom_alpha}
There exist functions $K$ and $\ell$ such that for every pair of
integers $k, \alpha \geq 1$, every digraph $D$ with independence
number $\alpha$ and dominating number at least $K(\alpha,k)$ contains
a $(k, \ell(\alpha,k))$-cluster.
\end{restatable}

\begin{corollary}\label{cor:uset_alpha}
There exist functions $K$ and $\ell$ such that for every pair of
integers $k, \alpha \geq 1$, every digraph $D$ with independence
number $\alpha$ contains either i) a dominating and an absorbing set,
each of size at most $K(\alpha,k)$, or ii) a
$(k,\ell(\alpha,k))$-cluster.
\end{corollary}

\begin{proof}
Let $k$ and $\alpha$ be constants, and $D$ a digraph with independence
number $\alpha$.  By Theorem \ref{thm:dom_alpha}, there exist
constants $K(\alpha,k)$ and $\ell(\alpha,k)$ such that one can find
either a dominating set of size at most $K(\alpha,k)$, or a subset of
size $\ell(\alpha,k)$ with chromatic number at least $k$.  Take the
digraph obtained by reversing all the arcs in $D$ and repeat the
previous argument. A dominating set in this digraph is an absorbing
set in $D$, while a subset with high chromatic number would also have
high chromatic number in $D$, as reversing all the arcs preserves the
chromatic number.
\end{proof}

In a digraph $D=(V,A)$, there are some pairs of vertices that do not
have arcs between them.  A pair $u,v$ is a {\em non-edge} in $D$ if
neither arc $uv$ nor arc $vu$ belongs to $A$.  The proof of the next
theorem involves adding arcs to a digraph $D$ to obtain a tournament.
Since there are two sets of arcs, $A$ and $B$, we use, for example,
$N_A^+(u)$ (rather than the more standard $N_D^+(u)$) and $N_B^+(u)$
to denote the set of vertices adjacent from $u$ via arcs in $A$ or
arcs in $B$, respectively.  We define $N_A^o(u)$ to be all vertices in
$V$ that form non-edges with $u$ in $D$.\footnote{We could have used $N_D^o(u)$ rather than $N_A^o(u)$ (since we never use $N_B^o(u)$), but we choose $N_A^o(u)$ for the sake of consistency.}  The goal in this section is
to prove our main theorem.

\boundedAlphaMain*

\begin{proof}
We prove this theorem by induction on $\alpha$.  For the base case,
Theorem \ref{thm:ours} proves the statement for $\alpha = 1$, by
setting $\h(t,1) = f(t)$.  For the induction hypothesis, we assume
that for any digraph $D=(V,A)$ with independence number $\alpha-1$, if
for all $e \in A$, $\chiv(D[N(e)]) \leq t$, then $\chiv(D) \leq
\h(t,\alpha-1)$.  Now our goal is to prove that for any digraph
$D=(V,A)$ with independence number $\alpha$, if for all $e \in A$,
$\chiv(D[N(e)]) \leq t$, then $\chiv(D) \leq \h(t,\alpha)$.

Consider a digraph $D=(V,A)$ with independence number $\alpha$.  We
construct a tournament $T = (V, A \cup B)$ where each arc in $B$ is a
non-edge in $D$.  Recall that we use $N_A^+(u)$ to denote the set of
vertices adjacent from $u$ via arcs from $A$.  Now we assign
directions as follows.  For each non-edge $u,v$ in $D$, if $N_A^+(v)
\cap N_A^-(u)$ and $N_A^+(u) \cap N_A^-(v)$ are both empty (i.e.,
contain no vertices) or are both non-empty, we direct the arc
arbitrarily. Otherwise either $N_A^+(u) \cap N_A^-(v) = \emptyset$,
and we direct the arc from $v$ to $u$, or $N_A^+(v) \cap N_A^-(u) =
\emptyset$ and we direct the arc from $u$ to $v$.  Thus, we have the
following property for each arc $uv$ in $B$: Either $N_A^+(v) \cap
N_A^-(u)$ contains no vertices or $N_A^+(v) \cap N_A^-(u)$ and
$N_A^+(u) \cap N_A^-(v)$ both contain at least one vertex.

Now our goal is to color the tournament $T$ such that each color class
induces an acyclic set of arcs from $A$. This will in turn bound the
chromatic number of $D$.  We use the notation $D[N_T(e)]$ to denote
the subgraph of $D$ (i.e., arcs from $A$) in the neighborhood of arc
$e$ in $T$.

\begin{claim}\label{clm:neighborhood}
$\forall e \in A \cup B, \chiv(D[N_T(e)]) \leq 3 \cdot \h(t,\alpha-1) +
  2t$. 
\end{claim}

\begin{cproof}
  Consider an arc $e=uv \in A$.  We partition $N_T(e)$ into three
  subsets of vertices.

  \begin{enumerate}
\item[(i)] $S_1=N_A^-(u) \cap N_A^+(v)$.  By the condition of the theorem,
$\chiv(D[S_1])=\chiv(D[N_A(e)]) \leq t$.
		
\item[(ii)] $S_2 = N_B^-(u)$. Then $D[S_2]$ has independence number at most
$\alpha -1$.  Thus, by the induction hypothesis, $\chiv(D[S_2]) \leq
\h(t,\alpha-1)$.
		
\item[(iii)] $S_3 = N_B^+(v)$. Then $D[S_3]$ has independence number at most
$\alpha -1$. Thus, by the induction hypothesis $\chiv(D[S_3]) \leq
\h(t,\alpha-1)$.

\end{enumerate}

Therefore, for an arc $e \in A$, we have $\chiv(D[N_T(e)]) \leq 2
\cdot \h(t, \alpha-1) + t$.  Next, we consider an arc $e = uv \in B$.
We partition $N_T(e)$ into three subsets of vertices.

  \begin{enumerate}
\item[(i)] $S_1=N_A^-(u) \cap N_A^+(v)$.  Then either $S_1$ is empty,
  in which case $\chiv(D[S_1])=0$, or $S_1$ is non-empty.  In this
  case, take any vertex $w \in N_A^+(u) \cap N_A^-(v)$. Notice that
  $S_1 \subseteq N_A(uw) \cup N_A(wv) \cup N_A^o(w)$.  By the
  condition of the theorem, $\chiv(D[N_A(uw)]) \leq t$ and
  $\chiv(D[N_A(wv)]) \leq t$.  Finally, $N_A^o(w)$ has independence
  number at most $\alpha-1$.  Thus by the induction hypothesis,
  $\chiv(D[N_A^o(w)]) \leq \h(t,\alpha-1)$. Therefore, $\chiv(D[S_1])
  \leq 2t + \h(t,\alpha-1)$.

\item[(ii)] $S_2 = N_B^-(u)$.  Then $D[S_2]$ has independence number at most
$\alpha -1$.  Thus, by the induction hypothesis $\chiv(D[S_2]) \leq
\h(t,\alpha-1)$.
 
\item[(iii)] $S_3 = N_B^+(v)$.  Then $D[S_3]$ has independence number at most
$\alpha -1$.  Thus, by the induction hypothesis $\chiv(D[S_3]) \leq
\h(t,\alpha-1)$.

\end{enumerate}

Therefore, $\chiv(D[N_T(e)]) \leq 3 \cdot \h(t,\alpha-1) + 2t$.
\end{cproof}

\begin{claim}\label{clm:paths}
For any pair of vertices $u,v$ in $V$, $\chiv(D[N_T^-(u)\cap
  N_T^+(v)]) \leq 15 \cdot \h(t,\alpha-1) + 10t$.
\end{claim}

\begin{cproof}
Recall that $D$, and therefore $T$, is strongly connected.  For any
pair of vertices $u,v$, take the shortest path $(e_i)_{1\leq i \leq
  k}$ from $u$ to $v$ in $T$.  Any vertex in $N_T^-(u)\cap N_T^+(v)$
must be in the neighborhood $N_T(e_i)$ of some arc $e_i$ of the
shortest path.  An arc from a vertex in $N_T(e_i)$ to a vertex in
$N_T(e_j)$ is {\em forward} if $i < j$ and {\em backward} if $j < i$.
There can be no arc in $A$ from a vertex in $N_T(e_i)$ to a vertex in
$N_T(e_j)$ for $j \geq i + 5$, or else there would be a shorter path
from $u$ to $v$. Thus, we can use five color palettes of $3 \cdot
\h(t,\alpha-1) + 2t$ colors each, and color $N_T(e_i)$ with the color
palette $i \bmod 5$.  By Claim \ref{clm:neighborhood}, each
neighborhood $N_T(e_i)$ does not contain a monochromatic directed
cycle of arcs from $A$.  Because all forward arcs from $A$ between
different neighborhoods are bicolored, this results in a coloring with
no monochromatic directed cycle of arcs from $A$.  In total, this
coloring uses $15 \cdot \h(t,\alpha-1) + 10t$ colors.
	\end{cproof}

If we can find a pair of vertices $u,v$ such that $\chiv(D[N_T^+(u)
  \cup N_T^-(v)])$ is small (i.e., bounded by a function of $t$ and
$\alpha$), then we can use Claim \ref{clm:paths} to bound $\chiv(D)$
and prove the theorem.  To do this, we need a few more tools.

\begin{claim}\label{clm:dom_bounded}
If the tournament $T = (V, A\cup B)$ has a dominating set
$\gammaout(T)$ and an absorbing set $\gammain(T)$, then $\chiv(D) \leq
|\gammaout(T)| \cdot |\gammain(T)| \cdot (15 \cdot \h(t,\alpha-1) +
10t + 2)$.
\end{claim}

\begin{cproof}
We now define a coloring $C$ of $D$. For each pair of vertices $u \in
\gammain(T), v \in \gammaout(T)$, we can color the set $N_T^-(u) \cap
N_T^+(v)$ using a different palette of $15 \cdot \h(t,\alpha-1) + 10t$
colors by Claim \ref{clm:paths}.  Each vertex $w$ of
$V\setminus{(\gammain(T) \cup \gammaout(T))}$ can be colored this way;
indeed for each such vertex $w$, there is some pair of vertices $u \in
\gammain(T), v \in \gammaout(T)$ such that $w \in N_T^-(u) \cap
N_T^+(v)$.  Moreover, each vertex in $\gammaout(T) \cup \gammain(T)$
can be colored with its own color.  If a vertex is assigned more than
one color, simply use the first color it is given.  This coloring uses
a total of at most $|\gammaout(T)| \cdot |\gammain(T)| \cdot (15 \cdot
\h(t,\alpha-1) + 10t) + |\gammaout(T)| + |\gammain(T)| \leq
|\gammaout(T)| \cdot |\gammain(T)| \cdot (15 \cdot \h(t,\alpha-1) +
10t + 2)$ colors.
\end{cproof}

Set $d = 3 \cdot \h(t, \alpha-1) + 2t$.  Notice that $T=(V,A \cup B)$
is not necessarily $d$-arc-bounded, since $\chiv(T[N_T(e)) \geq
  \chiv(D[N_T(e)])$.  We now modify the definition of a jewel (defined
  in the previous section) for our current setting: A {\em jewel} is a
  subset $J \subset V$ such that $J$ is a $(d+1, \ell(\alpha,
  d+1))$-cluster in $D$, so $\chiv(D[J]) \geq d+1$ and $|J| \leq \ell(\alpha, d+1)$.

\begin{definition}
We define a \emph{jewel-chain} in $T$ of length $p$ to be an ordered
set $X=(X_i)_{1\leq i \leq p}$ such that each $X_i$ induces a jewel
in $D$ (i.e., $D[X_i]$ is a jewel), all $X_i$'s are disjoint, and $X_i
\Rightarrow X_{i+1}$ for all $1\leq i \leq p-1$
(i.e., $X_i$ is complete to $X_{i+1}$ in $T$).
\end{definition}

As in the previous section, we say that for a jewel chain $X$, an arc
$uv$ is {\em forward} if $u \in X_i$ and $v \in X_j$ and $i < j$.  If
$j < i$, then arc $uv$ is {\em backward}.  The next claim is similar,
but not identical, to Observation \ref{obs:jewel}.  The subtle
difference stems from the fact that we care about the chromatic number
of jewel with respect to $D$ rather than $T$.

\begin{claim}\label{clm:backarc}
	A jewel-chain $X$ contains no backward arcs in $T$.
\end{claim}

\begin{cproof}
Consider a backward arc $e = uv$, with $u \in X_j$ and $v \in X_i$ for
$j > i$ such that $j-i$ is minimized.  It must be the case that $j >
i+1$, since all arcs between $X_i$ and $X_{i+1}$ are forward by
definition.  Then $X_{i+1} \subseteq N_T(e)$, and since
$\chiv(D[N_T(e)]) \geq d+1$, this contradicts Claim
\ref{clm:neighborhood}.  Thus, all arcs with endpoints in distinct
$X_i$'s must be forward.\end{cproof}

Let $X$ be a jewel-chain in $T$ of maximum length, say $p$.  Define
$Y$ to be the vertex set such that $Y \Rightarrow X_1$ in $T$.  Then
$D[Y]$ does not contain a jewel by assumption (otherwise, we could
make the jewel-chain longer).  By Corollary \ref{cor:uset_alpha},
since $D[Y]$ does not contain a $(d+1, \ell(\alpha, d+1))$-cluster,
$D[Y]$ contains a dominating set and an absorbing set, each of size at
most $K(d+1,\alpha)$.  Notice that a dominating (absorbing) set in
$D[Y]$ is also a dominating (absorbing) set in $T[Y]$.  So we can
apply Claim \ref{clm:dom_bounded} to bound the chromatic number of
$D[Y]$ by $(K(d+1,\alpha))^2 \cdot (15 \cdot \h(t,\alpha-1) + 10t +
2)$.

Moreover, the set $\npm_T(X_1)$ has chromatic number at most
$\chiv(D[\npm_T(X_1)]) \leq \ell(d+1,\alpha) \cdot d$.  Finally, $v
\in X_1$ can have in-neighbors in $X_1$ itself, but these can have
chromatic number at most $|X_1| \leq \ell(d+1,\alpha)$.

So for each vertex $v \in X_1$, we have
	$$\chiv(D[N_T^-(v)]) \leq (K(d+1,\alpha))^2 \cdot (15 \cdot \h(t,\alpha-1) + 10t + 2) + \ell(d+1,\alpha) \cdot (d+1).$$
By the same argument, each vertex $u \in X_p$ has the same bound on $\chiv(D[N_T^+(u)))$.
So we have
$$\chiv(D[N_T^+(u) \cup N_T^-(v)]) \leq 2((K(d+1,\alpha))^2 \cdot (15
\cdot \h(t,\alpha-1) + 10t + 2) + \ell(d+1,\alpha) \cdot (d+1)).$$
By Claim \ref{clm:paths}, we have
$$\chiv[D] \leq 2((1 + (K(d+1,\alpha))^2) \cdot (15 \cdot \h(t,\alpha-1) + 10t + 2) + \ell(d+1,\alpha) \cdot (d+1)).$$
Since $d = 3 \cdot \h(t, \alpha-1) + 2t$, we can define the function $\h$ as follows.

\begin{align*}
  \h(t,\alpha)  = ~ & 2((1+(K(3 \cdot \h(t, \alpha-1) + 2t+1,\alpha))^2) \cdot (15 \cdot \h(t,\alpha-1) + 10t + 2) \\ ~~& + \ell(3 \cdot \h(t, \alpha-1) + 2t+1,\alpha) \cdot (3 \cdot \h(t, \alpha-1) + 2t+1).
  \end{align*}
So we have $\chiv[D] \leq \h(t,\alpha)$, concluding the proof of the theorem.
\end{proof}

\setcounter{claim}{0}

\section{Equivalence of Conjectures \ref{conj:ee} and \ref{conj:big_to_big}}\label{sec:equiv}

\cite{nguyen2023some} show that Conjecture \ref{conj:big_to_big}
implies Conjecture \ref{conj:ee}.  In this section, we prove that
Conjecture \ref{conj:ee} implies Conjecture \ref{conj:big_to_big},
showing they are equivalent.  Our main tool is Theorem
\ref{thm:boundedAlpha}.

Let $\s$ be a function such that $\s(x) \geq x^2 \cdot \s(x-1)+x$ and
let $T$ be a tournament.  Recall that a \tfcluster is a subset $S$ of
$V$ of size $\s(t)$ such that $\chiv(T[S]) \geq t$.  For brevity, we
use $t$-cluster to denote a \tfcluster in this section.

\begin{definition}
Define a $t$-heavy arc $e \in A(T)$ to be an arc such that $T[N(e)]$
contains a $(t-1)$-cluster, and a $t$-light arc to be an arc that is
not $t$-heavy.
\end{definition}

Let us prove a lemma that will allow us to restate Conjecture
\ref{conj:big_to_big}.  The proof is reminiscent of the proof of 3.7
in \cite{berger2013tournaments} and essentially the same as the proof
of Lemma 3.4 in \cite{aboulker2022heroes}.  Let $\clust$ be a function
such that $\clust(x) = x \cdot 2^{\s(2x)}+\s(2x)+1$.

\begin{lemma}\label{lem:cluster_big}
For all $c \geq 0$, in any tournament $T$ with $\chiv(T) \geq
\clust(c)$ that has a $2c$-cluster, there are two sets $A, B \subseteq
V(T)$ such that $\chiv(T[A]), \chiv(T[B]) \geq c$ and all arcs between
$A$ and $B$ go from vertices of $A$ to vertices of $B$.
\end{lemma}

\begin{proof}
Let $C \subset V(T)$ be a $2c$-cluster.  By the definition of a
cluster, $|C| \leq \s(2c)$.  So there are at most $2^{\s(2c)}$ ways of
partitioning $C$. Consider any vertex $v \in V(T) \setminus{C}$.  Then
$(N^+(v) \cap C, N^-(v) \cap C)$ forms a partition of $C$.  Thus, we
can partition $V(T) \setminus{C}$ into at most $2^{\s(2c)}$ subsets
$(S_i)_{1\leq i \leq 2^{\s(2c)}}$ such that all vertices in a subset
$S_i$ have the same in-neighborhood and out-neighborhood in $C$ (i.e.,
each vertex in $S_i$ partitions $C$ in the same way).  If every $S_i$
can be colored with at most $c$ colors, $T$ can be colored with at
most $c \cdot 2^{\s(2c)}+\s(2c)$ colors.  Therefore, since $\chiv(T)
\geq \clust(c) = c \cdot 2^{\s(2c)}+\s(2c)+1$ by the condition of the
lemma, there must exist some subset $S_i$ with $\chiv(T[S_i]) \geq c$.
Consider the partition $(N^+(v) \cap C, N^-(v) \cap C)$ of $C$ for a
vertex $v \in S_i$.  Either $\chiv(T[N^+(v) \cap C]) \geq c$ or
$\chiv(T[N^-(v) \cap C]) \geq c$, since $\chi(C) \geq 2c$. By
definition, $S_i$ is complete to $N^+(v) \cap C$ and complete from
$N^-(v) \cap C$.  Thus by setting $A = N^-(v) \cap C$ and $B = S_i$ if
$\chiv(T[N^-(v) \cap C]) \geq c$, and $A = S_i$, $B = N^+(v) \cap C$
if $\chiv(T[N^+(v) \cap C]) \geq c$, we have found $A$ and $B$ with
$A$ complete to $B$ and $\chiv(T[A]), \chiv(T[B]) \geq c$.
\end{proof}

Let us restate Conjectures \ref{conj:ee} and \ref{conj:big_to_big}.

\begin{conjecture}[Restatement of Conjecture \ref{conj:ee}]\label{conj:ee2}
There exists a function $\ee$ such that for every pair of integers
$t,c \geq 1$, if a graph $G$ satisfies $\chi(G) \geq \ee(t,c)$ and
$\omega(G) < t$, then there are subsets $A,B \subseteq V(G)$ with
$\chi(G[A]), \chi(G[B]) \geq c$, such that there are no edges between
$A$ and $B$.
  \end{conjecture}

\begin{conjecture}\label{conj:big_to_big1}
There exists a function $\nss$ such that for every pair of integers
$t,c \geq 1$, if a tournament $T$ satisfies $\chiv(T) \geq \nss(t,c)$
and $T$ contains no $t$-cluster, then there are subsets $A,B \subseteq
V(T)$ such that $\chiv(T[A]), \chiv(T[B]) \geq c$ and all arcs between
$A$ and $B$ go from vertices in $A$ to vertices in $B$.
\end{conjecture}

Conjecture \ref{conj:big_to_big1} may seem weaker than Conjecture
\ref{conj:big_to_big}, but is in fact equivalent. This is a direct
consequence of Lemma \ref{lem:cluster_big}.  Indeed, for any $c$, if a
tournament $T$ has no sets $A$ and $B$ with $A$ complete to $B$ and
$\chiv(T[A]), \chiv(T[B]) \geq c$, then by the contrapositive of Lemma
\ref{lem:cluster_big} it has no $2c$-cluster or it has chromatic
number less than $\clust(c)$.  Therefore, Conjecture
\ref{conj:big_to_big1} will imply that $T$ has chromatic number
strictly less than $d=\max(\nss(2c,c),\clust(c))$, which is some
constant since $c$ is fixed. This is exactly the contrapositive of
Conjecture \ref{conj:big_to_big}.  We now state the contrapositive of
Conjecture \ref{conj:big_to_big1}, which is also equivalent to
Conjecture \ref{conj:big_to_big}.

\begin{conjecture}[Restatement of Conjecture
    \ref{conj:big_to_big1}]\label{conj:big_to_big2} There exists a
  function $\nss$ such that for every pair of integers $t,c \geq 1$,
  if a tournament $T$ contains no $t$-cluster and $T$ does not contain
  subsets $A,B \subseteq V(T)$ such that $\chiv(T[A]), \chiv(T[B])
  \geq c$ with all arcs between $A$ and $B$ going from vertices in $A$
  to vertices in $B$, then $\chiv(T) \leq \nss(t,c)$.
  \end{conjecture}

\begin{proof}[Proof of Conjecture \ref{conj:big_to_big2}, assuming
Conjecture \ref{conj:ee2}]

For $t=2$, a tournament $T$ with no $2$-cluster does not contain a
directed triangle and therefore has $\chiv(T) = 1$.  Thus, we have
$\nss(2,c) = 1$.  Now we assume that $\nss(t-1,c)$ exists.  We will
prove that $\nss(t,c)$ exists.

We consider a tournament $T$, which by assumption does not contain a
$t$-cluster.  Since $t$ is now fixed for the rest of this proof, we
simply use heavy and light in place of $t$-heavy and $t$-light.  Let
$L$ be the set of light arcs and $H$ the set of heavy arcs. Notice
that every arc in $T$ must be either in $L$ or in $H$.  Let $D_H =
(V,H)$ and $D_L = (V,L)$ be digraphs containing the heavy and light
arcs, respectively.  Let $G_H = (V,H)$ denote the undirected graph of
heavy edges and let $G_L = (V,L)$ denote the undirected graph of light
edges.  (Notice that we are abusing notation by using $H$ and $L$ to
refer to both directed and undirected edge sets.)

Our first claim is that the graph $G_H$ has no large clique,
and consequently, the graph $G_L$ has bounded independence
number.

\begin{claim}\label{clm:heavy_clique}
$\omega(G_H) \leq t-1$.
\end{claim}

\begin{cproof}
Suppose that $G_H$ contains a $K_t$ (i.e., a clique on $t$ vertices) and
let $S$ be the set obtained by including the $t$ vertices of the
clique in addition to the vertices in the $(t-1)$-cluster in the
neighborhood of each arc corresponding to an edge in the clique.  Then
$S$ has at most $t + t^2 \cdot \s(t-1)$ vertices.  Moreover, $T[S]$
cannot be colored with $t-1$ colors since every arc is heavy and the
endpoints of a heavy arc cannot have the same color in any coloring
using only $t-1$ colors.  Since $S$ contains a clique, we have that
$\chi(S) \geq t$.  Thus, $T$ contains a $t$-cluster, which is a
contradiction.\end{cproof}

\begin{claim}\label{clm:light_ind}
$\alpha(G_L) \leq t-1$.
\end{claim}
\begin{cproof}
$L$ and $H$ are complementary edge sets (i.e., every edge not in $L$
  belongs to $H$ and vice versa).  If $G_L$ has an independent set of
  size $t$, then $G_H$ would have a clique on those same $t$ vertices,
  which would contradict Claim \ref{clm:heavy_clique}.
\end{cproof}

\begin{claim}\label{clm:light_edge_neighb}
For every arc $e \in L$, $\chiv(T[N(e)]) \leq \nss(t-1,c)$.
\end{claim}
\begin{cproof}
By definition, the neighborhood of any light arc contains no
$(t-1)$-cluster.  Thus by the induction hypothesis it can be colored
with $\nss(t-1,c)$ colors.
\end{cproof}

It follows immediately that the neighborhood of every arc in $D_L$ has
chromatic number at most $\nss(t-1,c)$.  We can then use Theorem
\ref{thm:boundedAlpha} to show that $D_L$ can be colored with
$\h(\nss(t-1,c))$ of colors.

Fix such a coloring of $D_L$.  Each color induces a tournament that
has a vertex ordering in which each backward arc belongs to $H$
(since all monochromatic arcs with the same color from $L$ form an
acyclic digraph).  Consider the subtournament $T_i$ induced on
vertices with the $i^{th}$ color, let $n$ denote the number of
vertices in this subtournament and fix a vertex ordering $\{v_1,
\ldots, v_n\}$ in which all arcs in $D_L$ are forward.  Let $G_i$ be
the undirected graph on this vertex set whose edge set corresponds to
the backward arcs of $T_i$ with respect to the fixed vertex ordering.
Notice that $G_i$ is a subgraph of $G_H$, which is $K_t$-free by Claim
\ref{clm:heavy_clique}.

Now let us apply Conjecture \ref{conj:ee2} to the graph $G_i$.  Let
$c_2 = 2tc$.  Either each $G_i$ has chromatic number at most $d =
\ee(t,c_2)$ or $G_i$ contains two sets $S_1$ and $S_2$ with
$\chi(G[S_1]),\chi(G[S_2]) \geq c_2$ and with no edges in $G_i$
between $S_1$ and $S_2$. In the latter case, let $a$ be the smallest
index such that $\chi(G[\{v_1, \ldots, v_a\} \cap S_1]) \geq tc$, and
let $b$ be the smallest index such that $\chi(G[\{v_1, \ldots, v_b\}
  \cap S_2]) \geq tc$.  Without loss of generality, assume that $a <
b$.  Now let $A' = \{v_1, \ldots, v_a\} \cap S_1$ and $B' = \{v_{b+1},
\ldots, v_n\}\cap S_2$.  Observe that since $S_1$ and $S_2$ have no
arcs between them in $G_i$, which corresponds to the backedge graph of
$T_i$, then all arcs between $A'$ and $B'$ in $T_i$ must go from $A'$
to $B'$.  Moreover, we have
$\chiv(T_i[A']) \leq \chi(G_i[A']) \leq \omega(G_i[A'])
\chiv(T_i[A'])$.\footnote{This follows from 2.1 in \cite{nguyen2023some},
which says that $\chiv(T) \leq \chi(G) \leq \omega(G)\chiv(T)$ for a
backedge graph $G$ of tournament $T$.} Since $\chi(G_i[A']) \geq tc$
and $\omega(G_i[A']) \leq \omega(G_i) \leq t$, we have $\chiv(T_i[A'])
\geq c$.  Using the same argument, we also have $\chiv(T_i[B']) \geq
c$.  However, by assumption, such sets $A'$ and $B'$ do not exist in
$T$.  So we conclude that we are in the first case, in which
$\chiv(T_i) \leq \chi(G_i) \leq \ee(t,c_2)$.

Thus, we can color the subtournament induced by each color class of
$D_L$ with $\ee(t,2tc)$ colors, resulting in a coloring of $T$ with
$\nss(t,c) = \ee(t,2tc)\cdot \h(\nss(t-1,c))$ colors.
\end{proof}

\section*{Acknowledgements}
We thank Pierre Aboulker and Pierre Charbit for helpful conversations
and comments.  We thank Paul Seymour for pointing out that Theorem
\ref{thm:ours} is a special case of 13.3 in \cite{nguyen2023some}.  We
thank the anonymous referees for their help improving the presentation.

\newcommand{\etalchar}[1]{$^{#1}$}

\setcounter{claim}{0}

\appendix

\section{Proof of Theorem \ref{thm:dom_alpha}}\label{app:dom_lemma}

Let $D$ be a digraph with independence number $\alpha$, and let $X,Y
\subseteq V(D)$.  Then the following inequalities are straightforward.
\begin{eqnarray}
  \gamma(D[N^+[X]]) & \leq & |X|,\nonumber \\
  \gamma(D[Y]) & \leq & \gamma(D[X]) + \gamma(D[Y\setminus X]). \label{clm:inequalities}
\end{eqnarray}

\domLemma*

  \begin{proof}
  Let $P(\alpha,k)$ denote the statement of the theorem for $\alpha$
  and $k$.  Our goal is to prove $P(\alpha,k)$ for all integers
  $\alpha, k \geq 1$.  Let us assume that $P(\alpha-1,k)$ holds for
  all $k \geq 1$.  The base case for this is $P(1,k)$, which is proved
  in \cite{harutyunyan2019locToGlobal}.  Now we fix $\alpha$ and we
  want to prove $P(\alpha,k)$, which we will do by induction on $k$.
  The base case for this is $P(\alpha,1)$, which is true since any digraph with independence number $\alpha$ and domination number at least 1 contains at least one vertex, which serves as a $(1,1)$-cluster.  To build intuition, we can also consider the next case, which is $P(\alpha,2)$.  This is true since any
  digraph with independence number $\alpha$ and domination number at
  least $\alpha+1$ contains a directed cycle of length at most
  $\ell(\alpha,2) \leq 2\alpha+1$, and this cycle requires two
  colors.\footnote{This follows from the well-known classical theorem
  that an acyclic digraph has an independent dominating set.  See
  \cite{bondy2003short}.}  Now we assume $P(\alpha,k-1)$ (as well as
  $P(\alpha-1,k)$) and we want to prove $P(\alpha,k)$.
  
We will follow the proof of Theorem 5 from \cite{harutyunyan2019locToGlobal}. Let us first prove a useful claim.  Recall that $N^o(v)$ is the set of vertices that form non-edges with $v$.

\begin{claim}\label{clm:non_neighborhood}
If $D$ does not contain a $(k, \ell(\alpha-1,k))$-cluster,
  then for any vertex $v \in V(D)$,
\begin{equation*} \gamma(D[N^o(v)]) \leq K(\alpha-1,k). \end{equation*}
\end{claim}

\begin{cproof}
The digraph $D' = D[N^o(v)]$ has independence number $\alpha-1$.  By
the inductive hypothesis on $\alpha$, either $D'$ has a $(k,
\ell(\alpha-1,k))$-cluster or $D'$ has domination number at most
$K(\alpha-1,k)$.  Thus, $\gamma(D[N^o(v)]) \leq K(\alpha-1,k)$.
\end{cproof}

Let $D=(V,E)$ be a digraph with independence number $\alpha$ such that
$\gamma(D) \geq K(\alpha,k)$, and let $B$ be a minimum dominating set
of $D$.  We will assume that $D$ does not contain a $(k,
\ell(\alpha-1,k))$-cluster, since otherwise, we would be done.
Fix $$K(\alpha,k) = k
(K(\alpha-1,k) + 1)(K(\alpha,k-1)+\ell(\alpha,k-1)\cdot
(K(\alpha-1,k)+1)+1)+K(\alpha,k-1).$$
Consider a subset
$W$ of $B$, where
$$|W| = k(K(\alpha,k-1)+\ell(\alpha,k-1)\cdot
(K(\alpha-1,k)+1)+1).$$
From \eqref{clm:inequalities} and
Claim \ref{clm:non_neighborhood}, we have
\begin{eqnarray*}
\gamma(D[V \setminus (N^+[W] \cup
N^o(W))]) & \geq & \gamma(D) - \gamma(D[N^+[W]]) - \gamma(D[N^o(W)])\\ & \geq &
\gamma(D) -|W| - |W|(K(\alpha -1,k)\\
& \geq &
K(\alpha,k)-|W|(K(\alpha-1,k)+1)\\ & \geq & K(\alpha,k-1).
\end{eqnarray*}

By applying the induction hypothesis, the digraph
$D[V\setminus{(N^+[W] \cup N^o(W))}]$ contains a $(k-1,
\ell(\alpha,k-1))$-cluster.  Call this vertex set $A$.  Note that by
construction, $A \cap W = \emptyset$ and $A$ is complete towards $W$.
Now consider a subset $S$ of $W$ where
$$|S| = K(\alpha,k-1)+\ell(\alpha,k-1)\cdot (K(\alpha-1,k)+1)+1.$$
We claim
that
\begin{eqnarray}
\gamma(D[N^+(S)]) \geq K(\alpha,k-1) +
\ell(\alpha,k-1)\cdot(K(\alpha-1,k)+1).\label{Splus}
\end{eqnarray}
If not, we can choose a
dominating set $S'$ of $N^+(S)$, where
$$|S'| \leq K(\alpha,k-1)+\ell(\alpha,k-1)\cdot(K(\alpha-1,k)+1)-1.$$  Note that
$x$ dominates $S$ for any $x \in A$, and so $S' \cup \{x\}$ dominates
$N^+[S]$. Hence $(B\setminus S) \cup S' \cup \{x\}$ would be a
dominating set of $D$ of size less than $|B|$ which contradicts the
minimality of $B$.  We therefore conclude that Inequality \eqref{Splus} holds.

	Let $N' = N^+(S) \setminus (N^+(A) \cup N^o(A))$. From Claims 
	\ref{clm:inequalities} and \ref{clm:non_neighborhood} we have
        \begin{eqnarray*}
	\gamma(D[N']) & \geq & \gamma(D[N^+(S)]) - \gamma(D[N^+(A)]) - \gamma(D[N^o(A)])\\ & \geq&  K(\alpha,k-1) +
	\ell(\alpha,k-1)\cdot(K(\alpha-1,k)+1) - |A|(K(\alpha-1,k)+1)\\ & =& K(\alpha,k-1).
        \end{eqnarray*}
Thus, by the induction hypothesis on $k$, there is a subset $A_s
\subseteq N'$ that forms a $(k-1, \ell(\alpha, k-1))$-cluster.
By construction, $A_S \cap A = \emptyset$ and $A_S$ is complete
towards $A$.

We now construct our subdigraph of $D$ with chromatic number at
least $k$.
We consider the set of vertices $A \cup W$ to
which we add the collection $A_S$, for all subsets $S \subseteq W$
of size $K(\alpha,k-1)+\ell(\alpha,k-1)\cdot
(K(\alpha-1,k)+1)+1$. Call $A'$ this new vertex set and observe that
its size is at most
$$|A'| \leq |A| + |W| + |A_S| {|W| \choose |S|}.$$
So we have
\begin{eqnarray*}
\ell(\alpha,k)  = & & \ell(\alpha,k-1) +
k(K(\alpha,k-1)+\ell(\alpha,k-1)\cdot (K(\alpha-1,k)+1)+1) \\ & +&
\ell(\alpha,k-1)\binom{k(K(\alpha,k-1)
   + \ell(\alpha,k-1)\cdot
  (K(\alpha-1,k)+1)+1)} {K(\alpha,k-1)+\ell(\alpha,k-1)\cdot
  (K(\alpha-1,k)+1)+1}.
\end{eqnarray*}
To conclude, it is sufficient to show that $\chi(A') \geq k$. Suppose
not, and for contradiction, take a $(k-1)$-coloring of $A'$. Since
$|W| = k(K(\alpha,k-1)+\ell(\alpha,k-1)\cdot (K(\alpha-1,k)+1)+1)$
there is a monochromatic set $S$ in $W$ of size
$K(\alpha,k-1)+\ell(\alpha,k-1)\cdot (K(\alpha-1,k)+1)+1$ (say,
colored 1). Recall that $A_S$ is complete to $A$, and $A$ is complete
to $S$, and note that since $\chi(A) \geq k-1$ and $\chi(A_S) \geq
k-1$, both $A$ and $A_S$ have a vertex of each of the $k-1$
colors. Hence there are $u \in A$ and $w \in A_S$ colored 1. Since
$A_S \subseteq N^+(S)$, there is $v \in S$ such that $(v,w)$ is an arc
of $D$. We then obtain the monochromatic triangle $(u,v,w)$ of color
1, a contradiction. Thus, $\chiv(D[A']) \geq k$ implying that $A'$ is
a $(k,\ell(\alpha,k))$-cluster in $D$ completing the induction
on $k$.

Since this induction proves the statement $P(\alpha,k)$ holds for any $k$, it
proves the inductive hypothesis for $\alpha$.
Then, by induction on $\alpha$ we have proven that the theorem is true
for any pair of integers $\alpha, k$.
\end{proof}


\end{document}